\documentclass[a4paper]{amsart}

\usepackage{amsthm}
\usepackage{amsfonts}
\usepackage{amsmath,amsthm,amssymb}
\usepackage{comment}
\usepackage{fancybox}
\usepackage{epsf}
\usepackage[dvips]{graphicx}
\usepackage[all]{xy}

\newtheorem{theorem}{Theorem}[section]
\newtheorem{lemma}[theorem]{Lemma}
\newtheorem{proposition}[theorem]{Proposition}
\newtheorem{corollary}[theorem]{Corollary}
\theoremstyle{definition}
\newtheorem{definition}[theorem]{Definition}
\theoremstyle{remark}
\newtheorem{remark}[theorem]{Remark}
\newtheorem{example}[theorem]{Example}

\begin{document} 

\title[Finite type invariants for cyclic equivalence classes of nanophrases]{Finite type invariants for cyclic equivalence classes of nanophrases}
\author[Yuka Kotorii]{Yuka Kotorii}

\address{
Department of Mathematics \\
Tokyo Institute of Technology \\
Oh-okayama \\
Meguro \\
Tokyo 152-8551 \\
Japan
}

\email{kotorii.y.aa@m.titech.ac.jp}



\thanks{}

\begin{abstract} 

In this paper, we define finite type invariants for cyclic equivalence classes of nanophrases
and construct the universal ones.
Also, we identify the universal finite type invariant of degree 1 essentially with the linking matrix.
It is known that extended Arnold's basic invariants to signed words are finite type invariants of degree 2, by Fujiwara.
We give another proof of this result and show that those invariants do not provide the universal one of degree 2.

\end{abstract} 

\date{\today}

\maketitle

\section{Introduction}
Turaev developed the theory of words based on the analogy with curves on the plane, knots in the 3-sphere, virtual knots, etc. in \cite{MR2276346,MR2352565}.  
A word is a sequence of letters, belonging to a given set, called alphabet.
Let $\alpha$ be a set.
A nanoword over $\alpha$ is a pair of a word in which each letter appears exactly twice 
and a map from the set of letters appearing in the word to $\alpha$, defined by \cite{MR2352565}.
A nanophrase is a generalization of nanoword, defined by \cite{MR2276346}.
   
Vassiliev developed in \cite{MR1089670} the theory of finite type invariants of knots, which is conjectured to classify knots.
In \cite{MR2573961}, Ito defined a notion of finite type invariants for curves on surfaces, constructed a large family of finite type invariants, called $SCI_m$, 
and showed that they become a complete invariant for stably homeomorphism classes.
On the other hand, in \cite{label8048} Fujiwara provided a simple idea to define finite type invariants for cyclic equivalence classes of signed words 
by introducing a new type of crossing, 
called singular crossing, which plays intermediate role between an actual and virtual crossing.
Here the signed word is a nanoword over $\alpha=\{ +,-\}$ and it is known in \cite{MR2276346} that the set of cyclic equivalence classes of signed words bijectively corresponds to the set of stably homeomorphism classes of curves on surfaces.
We extend Fujiwara's finite type invariants,
to those for cyclic equivalence classes of nanophrases over a general $\alpha$ not necessarily to be $\{ +, - \}$. 
The first purpose of this paper is to construct the universal ones in Theorem \ref{f.thm}, by following the approach in \cite{MR1763963}.
In addition, we identify the universal finite type invariant of degree 1 essentially with the linking matrix.
As a related work, we should mention that Gibson and Ito extended the universal finite type invariant for nanophrases under different equivalence relations, called homotopy and closed homotopy, in \cite{MR2786675}.

To see the second purpose, recall that in \cite{MR1650406}, Polyak reconstructed Arnold's basic invariants \cite{MR1310595, MR1286249} for isotopy classes of generic curves on the plane or sphere 
by using Gauss diagrams and showed that those invariants are finite type invariants of degree 1 in his sense.
In \cite{MR2573961}, Ito connected Arnold's basic invariants of planar curves with his finite type invariants.
On the other hand, in \cite{label8048} Fujiwara extended Arnold's basic invariants of spherical curves to signed words 
and showed that those are finite type invariants of degree 2 in his sense.
We give another proof of Fujiwara's result and show that they do not provide the universal invariant of degree 2 in Fujiwara's sense.

This paper is organized as follows.
In section 2, following Turaev, we give formal definitions of words, phrases and so on.
In section 3, following Fujiwara, we give definitions of singular crossings and singular letters. 
In section 4, we define finite type invariants for cyclic equivalence classes of nanophrases.
In section 5, we construct the universal finite type invariants for cyclic equivalence classes of nanophrases.       
In section 6, we identify the universal finite type invariant of degree 1 essentially with the linking matrix.
In section 7, we restrict our discussion to nanowords corresponding to spherical curves,
and clarify the relation between Arnold's basic invariants and our finite type invariants. 
\section{Nanowords and nanophrases}

In this section, following Turaev \cite{MR2276346,MR2352565}, 
we review formal definitions of words, phrases and so on. 

\subsection{Words and phrases}
An {\it alphabet} is a finite set and its element is called a {\it letter}.  
A {\it word} of length $m$ is a finite sequence of $m$ letters.
The unique word of length 0 is called the {\it trivial word} and is written by $\emptyset$. 
An $n$-component {\it phrase} is a sequence of $n$ words, each of which we call a {\it component}.  
We write a phrase as a sequence of words, separated by `$|$'.  
The unique $n$-components phrase for which every component is the trivial word is called the {\it trivial $n$-component phrase} and is denoted by $\emptyset_n$.  
In this paper, we will regard words as 1-component phrases.

\subsection{Nanowords and nanophrases}
Let $\alpha$ be a finite set.  
An $\alpha$-alphabet is an alphabet $\mathcal{A}$ together with an associated map from $\mathcal{A}$ to $\alpha$.  
This map is called a {\it projection}.  
The image of any $A \in \mathcal{A}$ in $\alpha$ will be denoted by $|A|$.
An {\it isomorphism} $f$ of an $\alpha$-alphabet $\mathcal{A}_1$ to $\mathcal{A}_2$ is a bijection 
such that $|f(A)|$ is equal to $|A|$ for any letter $A$ in $\mathcal{A}_1$.
A {\it Gauss word}  on an alphabet $\mathcal{A}$ is a word on $\mathcal{A}$ such that every letter 
in $\mathcal{A}$ appears exactly twice.  
Similarly, a {\it Gauss phrase} on $\mathcal{A}$ is a 
phrase which satisfies the same condition.  
By definition, a 1-component Gauss phrase is a Gauss word. 

Let $p$ be a Gauss word or a Gauss phrase.  
The {\it rank} of $p$ is the number of distinct letters appearing in $p$.  
We denote it by ${\rm rank}(p)$. 
Note that the rank of a Gauss word must be a half of its length.
For example, the rank of $ABCBAC$ is $3$ and the rank of $A|B|\emptyset |BA$ is $2$.

An $n$-component {\it nanophrase} over $\alpha$ is a pair $(\mathcal{A}, p)$ 
where $\mathcal{A}$ is an $\alpha$-alphabet and $p$ is an $n$-component Gauss phrase on $\mathcal{A}$.
When $n$ is equal to 1, we call it a {\it nanoword}.  

Let $(\mathcal{A}_1, p_1)$ and $(\mathcal{A}_2, p_2)$ be nanophrases over $\alpha$.
An {\it isomorphism} $f$ of a nanophrase $(\mathcal{A}_1, p_1)$ to $(\mathcal{A}_2, p_2)$ is an isomorphism $f$ of $\alpha$-alphabets such that $f$ applied letterwisely to the $i$th component of $p_1$ gives the $i$th component of $p_2$ for all $i$.
If such an isomorphism exists, we say that $(\mathcal{A}_1, p_1)$ and $(\mathcal{A}_2, p_2)$ are isomorphic.

We can define the rank of a nanoword and a nanophrase similar to that of a Gauss word and a Gauss phrase.


\subsection{Shift move on nanowords and nanophrases}
In \cite{MR2276346}, Turaev defined a {\it shift move} on nanophrases.
Let $\nu $ be an involution on $\alpha $.
Suppose $p$ is an $n$-component nanophrase over $\alpha $.
A  {\it $\nu $-shift move} on the $i$th component of $p$ is a move which gives a new nanophrase $p' $ as follows.
If the $i$th component of $p$ is empty or only a single letter, then $p'$ is $p$.
If not, we can write the form of the $i$th component of $p$ by $Ax$.
Then the $i$th component of $p'$ is $xA$ and other components of $p'$ are the same as the corresponding components of $p$.
Furthermore, if we write $|A|_p$ for $|A|$ in $p$ and $|A|_{p'}$ for $|A|$ in $p'$, then $|A|_{p'}$ equals $\nu(|A|_{p})$ 
when $x$ contains the letter $A$ and otherwise, $|A|_{p'}$ equals $|A|_p$.


\section{Singular crossings and singular letters on signed words}

In \cite{label8048}, Fujiwara introduced singular letters on signed words, which are nanowords over $\alpha=\{ +,- \}$.
In this section, following Fujiwara, we review definitions of singular crossings and singular letters on signed words.

\subsection{Curves and signed words}

In this paper, a curve means a generic immersion from an oriented circle to a closed oriented surface.
Here generic means that the curve has only a finite set of self crossings which are all transversally double points, 
does not have triple points and self tangencies, and has a regular neighborhood.
A pointed curve is a curve endowed with a base point away from the double points.

Let $\alpha $ be the set $\{+,- \}$ and $\mathcal{A}$ an $\alpha$-alphabet.
We then call a nanoword $(\mathcal{A},w)$ over $\alpha$ a signed word.

We consider a pointed curve on a surface.  
We label these crossings in an arbitrary way by different letters $A_1,A_2,...,A_m$, 
where $m$ is the number of crossings.
The Gauss word of a curve is obtained by the following.
We start from the base point, move along the curve following the orientation and finish when we get back to the base point.
Then we write down all the letters of the crossings in the order we meet them.
The resulting word $w$ on the alphabet $\mathcal{A}=\{ A_1,A_2,...,A_m\}$ contains each letter twice.
Making $w$ a signed word, we consider the crossing of the curve labeled by $A$.
When moving along the curve as above, if we first traverse this crossing from the bottom-left to the top-right,
then $|A|=-$, otherwise $|A|=+$, see Figure \ref{fig:cross}.
Here the orientation of the ambient surface is counterclockwise.
The trivial curve is corresponding to the signed word $\emptyset$.
If we choose a different choice of the labeling of the crossing, we get an isomorphic signed word. 
We assign to each curve on a surface the isomorphism class of this signed word.

Two curves are stably homeomorphic 
if there is a homeomorphism between their regular neighborhoods 
which maps one curve onto the other preserving the orientation of the curves and the ambient surfaces.  
Similarly, two pointed curves are pointed stably homeomorphic 
if two curves are stably homeomorphic mapping one base point onto the other.

If we change the curve by a stably homeomorphism, then the associated signed word does not change, since it is defined by the behavior of the curve in its regular neighborhood.
It is proved in \cite{MR2276346} that the set of isomorphism classes of signed words bijectively corresponds to the set of stably homeomorphism classes of pointed curves on surfaces.

Let $\nu$ be an involution on $\alpha$ which maps $+$ to $-$.
The {\it cyclic equivalence relation} on signed words is defined by the relation generated by the $\nu$-shift move.
It is proved in \cite{MR2276346} that the set of cyclic equivalence classes of signed words bijectively corresponds to the set of stably homeomorphism classes of curves on surfaces.

\begin{figure}[h]
\begin{center}
\includegraphics[scale=0.7]{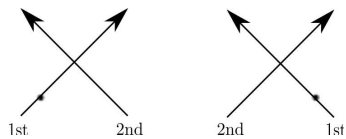}
\caption{On the left the sign is $-$, and on the right $+$}
\label{fig:cross}
\end{center}
\end{figure}
\subsection{Singular curves and singular signed words}

When we project a curve on a surface to the plane,
Some crossings which are not crossings on a surface may appear in the planar curve.
We call such a crossing a virtual crossing and denote it by a crossing with a small circle surrounding it. 
Such crossings do not contribute at all to the associated signed word.
Virtual crossing is not just an ordinary graphical vertex,
but an non-actual crossing.

Let $c$ be a curve on a surface, and $w_c$ a signed word corresponding to $c$.
If we replace some crossings of the planar curve which come from actual crossings of $c$ by virtual crossings,
we get a new curve on some surface. 
The signed word corresponding to the new curve is obtained from $c$ by deleting letters assigned to the actual crossing in question.

Fujiwara \cite{label8048} introduced a new type of crossing called {\it singular} crossing, 
which is an intermediate crossing between an actual crossing and a virtual crossing, 
and denoted it as a crossing with a small box surrounding it.  
A singular curve is a planar curve which may contain singular crossings.

We call a signed word corresponding with a singular curve a singular signed word.
Getting signed words corresponding with singular curves, 
we label actual and singular crossings by the letters such as $A$ and the letters with asterisk such as $A^*$, respectively.
The letters such as ${A}^*$ are called {\it singular} letters.
For a word $w$ without singular letters, 
denote by $w^*$ the word obtained by changing all letters of $w$ to singular letters saving signs.

\begin{figure}[h]
\begin{center}
\includegraphics[scale=0.3]{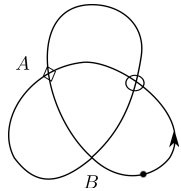}
\caption{The singular curve corresponding to the signed word ${A}^*B{A}^*B$, where $|{A}^*|=+$ and $|B|=-$}
\label{fig:singular}
\end{center}
\end{figure}
\bigskip

\section{Definitions of finite type invariants}

In \cite{label8048}, Fujiwara defined finite type invariants for cyclic equivalence classes of signed words.
We now extend the definition of finite type invariants for general $\alpha $.

Let $\alpha $ be any finite set.
Let $P(\alpha,n)$ denote the set of isomorphism classes of $n$-component nanophrases over $\alpha $. 
Let $\nu $ be any involution on $\alpha$.
We define the {\it cyclic equivalence relation} over $n$-component nanophrases as equivalence relations generated by isomorphisms and $\nu$-shift moves.
Let $P(\alpha ,\nu ,n)$ denote the set of cyclic equivalence classes of $n$-component nanophrases over $\alpha $.

We define $\alpha^*= \{a^* \mid a \in \alpha \}$ and let $\mathcal{A}$ be an $\alpha\cup \alpha^*$-alphabet.
Then the letter whose projection is contained in $\alpha^\ast $ is called a {\it singular letter} and is denoted with asterisk, say by $A^\ast $, for easy distinction.
Let $P_m(\alpha ,n)$ denote the set of isomorphism classes of $n$-component nanophrases over $\alpha \cup \alpha ^*$ with $m$ singular letters.
We call the phrase with singular letters a {\it singular phrase}.
By definition, $P_0(\alpha ,n)$ is the set of non-singular $n$-component nanophrases and so $P_0(\alpha ,n)=P(\alpha,n)$.

Given an involution $\nu $ over $\alpha $, we extend $\nu $ to $\alpha \cup \alpha ^*$ as follows.
For any $a\in \alpha $, we define 
\[\nu (a^\ast )=\nu (a)^\ast. \]  
Then let $P_m(\alpha ,\nu ,n)$ denote the set of cyclic equivalence classes of $n$-component nanophrases over $\alpha \cup \alpha ^*$ with $m$ singular letters, where $m$ is at least 0.
By definition, $P(\alpha ,\nu ,n) = P_0(\alpha ,\nu ,n)$.
Put $\mathcal{P}(\alpha ,\nu ,n)=\cup_{m\geq 0}P_m(\alpha ,\nu ,n)$. 

An invariant $u$ for cyclic equivalent classes of nanophrases is a map from the set of $n$-component nanophrases in an abelian group $G$, 
which takes equal values on nanophrases related by isomorphisms and $\nu$-shift moves.  
In other words, it is a map $u:P(\alpha ,\nu ,n) \rightarrow G$.

Given an invariant $u:P(\alpha ,\nu ,n) \rightarrow G$, 
we define its extension $\hat{u} :\mathcal{P}(\alpha ,\nu ,n) \rightarrow G$ by the following rule
\begin{eqnarray*}
\left\{
\begin{array}{l}
\hat{u}(p) = u(p) \hspace{4cm} \text{ if } p\in P_0(\alpha ,\nu ,n) \\  
\hat{u}(xA^\ast yA^\ast z) = \hat{u}(xAyAz)-\hat{u}(xyz) \hspace{4mm} \text{ if } xA^\ast yA^\ast z\in P_m(\alpha ,\nu ,n) ~~ (m\geq 1),
\end{array}
\right.
\end{eqnarray*}
where $A$ is a non-singular letter such that $|A|^\ast $ is equal to $|A^\ast |$ and $x$, $y$ and $z$ are arbitrary sequences of letters possibly including `$|$' or  `$ \emptyset $'. 
This map is well defined because the result does not depend on the order of the singular letters which we exclude.
That is, the following two calculations have the same result.
\begin{align*}
\hat{u}(xA^\ast yB^\ast  & zA^\ast wB^\ast t) \\
&= \hat{u}(xA yB^\ast zA wB^\ast t) - \hat{u}(x yB^\ast z wB^\ast t) \\
&= \hat{u}(xA yB zA wB t) - \hat{u}(xA y zA w t) - \hat{u}(x yB z wB t) + \hat{u}(x y z w t), \\
\hat{u}(xA^\ast yB^\ast  & zA^\ast wB^\ast t) \\
&= \hat{u}(xA^\ast yB zA^\ast wB t) - \hat{u}(xA^\ast y zA^\ast w t) \\
&= \hat{u}(xA yB zA wB t) - \hat{u}(x yB z wB t) - \hat{u}(xA y zA w t) + \hat{u}(x y z w t).
\end{align*}

Let $p$ and $q$ be $n$-component nanophrases.
A nanophrase $q$ is a {\it subphrase} of $p$, denoted by $q\triangleleft p$, if it is a nanophrase obtained from $p$ by deleting pairs of letters.
Here each letter does not change the value of the projection.
We use this word even if we eliminate no-letters. 
If the rank of $p$ is $k$, then $p$ has exactly $2^k$ subphrases.

\begin{example}
Let $p=ABA|B$. Then the subphrases of $p$ are $ABA|B$, $AA|\emptyset$, $B|B$ and $\emptyset|\emptyset$.
\end{example}

For any nanophrases $p$ and $q$, we define $\delta{(p,q)}$ by
\[ \delta{(p,q)}=\text{rank}(p)-\text{rank}(q). \]

\begin{proposition}\label{f.pro0}
For any $n$-component nanophrases $p$ in $\mathcal{P}(\alpha ,\nu ,n)$ and any invariant $u:P(\alpha ,\nu ,n) \rightarrow G$, 
\begin{align}
\hat{u}(p)= \sum_{p''\triangleleft q \triangleleft p'} (-1)^{\delta{(p',q)}}u(q),   \label{f.in prop0}
\end{align} 
where $p'$ is the non-singular phrase obtained by replacing all singular letters of $p$ by the corresponding non-singular letters 
and $p''$ is the subphrase of $p$ obtained from $p$ by deleting all singular letters.
\end{proposition}
\begin{proof}
If $p$ does not have singular letters, then this assertion is trivial.
If $p$ has some singular letter $A^\ast$, then we can write that $p'=xAyAz$, where $x$, $y$ and $z$ are non-singular phrases and $p''$ is a subphrase of $xyz$.
Therefore the right hand side of the identity (\ref{f.in prop0}) can be written by the following formula in the additive abelian group $\mathbb{Z}P(\alpha ,\nu ,n)$ freely generated by $P(\alpha ,\nu ,n)$
\begin{align*} 
 \sum_{p''\triangleleft q \triangleleft xAyAz} (-1)^{\delta{(xAyAz,q)}} \hat{u}(q). 
\end{align*} 
We divide this sum into two sums according to whether the phrase $q$ contains the letter $A$ or not, 
and then use the extension rule of $\hat{u}$.
For any subphrases $p_1$ and $p_2$ of $p$, let $p_1 \cup_p p_2$ denote the subphrase obtained from $p$ by deleting the letters contained in neither $p_1$ nor $p_2$.
Then the above sum is divided into 
\begin{align*} 
& \sum_{(p''\cup_{xAyAz} AA) \triangleleft q \triangleleft xAyAz} (-1)^{\delta{(xAyAz,q)}}\hat{u}(q) + \sum_{p'' \triangleleft q \triangleleft xyz} (-1)^{\delta{(xyz,q)}}\hat{u}(q) \\
&~~~~~~~~~~~= \sum_{(p''\cup_p A^*A^*) \triangleleft q \triangleleft xA^*yA^*z} (-1)^{\delta{(xA^*yA^*z,q)}}\hat{u}(q). 
\end{align*} 
Repeating this simplification and we have 
\begin{align*}
\sum_{p \triangleleft q \triangleleft p} (-1)^{\delta{(p,q)}}\hat{u}(q) = \hat{u}(q). 
\end{align*} 
\end{proof}

A map $u :P(\alpha ,\nu ,n) \rightarrow G$ is a {\it finite type invariant} 
if there exists a non-negative integer $m$ such that for any $n$-component nanophrases $p$ with more than $m$ singular letters, $\hat{u}(p)$ is zero.
The minimal such $m$ is called the {\it degree} of $u$.

The map $u:P(\alpha ,\nu ,n) \rightarrow G$ is a {\it universal finite type invariant} of degree $m$
if for any finite type invariant $v$ of degree less than or equal to $m$ taking values in some abelian group $H$,
there exists a homomorphism $f$ from $G$ to $H$ such that 
\[ \xymatrix{  
P(\alpha ,\nu ,n) \ar[r]^(0.6)u   \ar[dr]_(0.6)v  &  \ar@{}[dl]|(0.3)\circlearrowright G  \ar[d]^f \\
 & H 
 } \]

In particular, if $p$ and $q$ are two $n$-component nanophrases over $\alpha $ which can be distinguished by a finite type invariant of degree less than or equal to $m$ and $u$ is the universal invariant of degree $m$, then $u(p)$ is not equal to $u(q)$.

In \cite{MR1763963}, Goussarov, Polyak and Viro constructed the universal finite type invariants of virtual knots and links.
In \cite{MR2786675}, Gibson and Ito constructed the universal finite type invariants for homotopy classes of nanophrases.
In a similar way, we construct the universal finite type invariants for cyclic equivalence classes of nanophrases in the next section.

\section{Universal finite type invariants}

In this section, following the approach by Goussarov, Polyak and Viro in \cite{MR1763963} we construct the universal finite type invariants for cyclic equivalence classes of nanophrases over general $\alpha$.

Let $\mathbb{Z}P(\alpha,n)$ be the additive abelian group freely generated by isomorphism classes of $n$-component nanophrases. 
Let $\mathbb{Z}P(\alpha , \nu ,n)$ be the additive abelian group freely generated by $P(\alpha , \nu ,n)$.

Let $G(\alpha ,\nu ,n)$ be the group obtained from $\mathbb{Z}P(\alpha, n)$ by taking quotient with the following relations,
\begin{equation*}
\begin{cases}
w|AxAy|t-w|xByB|t=0 \\
w|Az|t-w|zA|t=0,
\end{cases}
\end{equation*}
where $|B|=\nu (|A|)$, $x$, $y$ and $z$ do not contain `$|$' and $z$ does not contain $A$.

The projection from $\mathbb{Z}P(\alpha, n)$ to $\mathbb{Z}P(\alpha , \nu ,n)$ is a homomorphism, and we denote it by $I$. 

\begin{proposition}
The homomorphism $I$ induces an isomorphism $\hat{I}: G(\alpha, \nu, n) \rightarrow \mathbb{Z}P(\alpha , \nu ,n)$. 
\end{proposition}
\begin{proof}
It is clear that $I$ is onto and the subgroup generated by the element $w|AxAy|t-w|xByB|t$ and $w|Az|t-w|zA|t$ in $\mathbb{Z}P(\alpha, n)$ is a subset of the kernel of $I$.
Contrarily, we show that the kernel of $I$ is a subset of the subgroup in question.
For any $p$ in $\mathbb{Z}P(\alpha, n)$, there exists a finite subset $X$ of $P(\alpha, n)$ such that $p=\sum_{x \in X} e_x x$,
where $e_x$ is an integer.
If $p$ is contained in the kernel of $I$, we have 
\[0= I \bigl(\sum_{x\in X}e_x x \bigr) =\sum_{x\in X} e_x[x] =\sum_{x\in X'} \bigl(\sum_{x\sim y, y\in X}e_y \bigr) [x], \]
where $X'$ is a set of representatives. 
It is true if and only if for any $x \in X'$, $\sum_{x\sim y, y\in X}e_y=0$. 
Therefore $\sum_{x\sim y, y\in X}e_y y$ is contained in the above subgroup and so is $p$.
\end{proof}

We then define a homomorphism $\theta _n : \mathbb{Z}P(\alpha, n) \rightarrow \mathbb{Z}P(\alpha, n)$ as follows.
For any $n$-component nanophrase $p$, $\theta _n(p)$ is the sum of all the subphrases of $p$ considered as an element of $\mathbb{Z}P(\alpha, n)$.
We then extend $\theta _n$ linearly to all of $\mathbb{Z}P(\alpha, n)$.
Note that for any nanophrase $p$,  $\theta _n(p)$ can be written as 
\[\theta _n(p) =\sum_{q\triangleleft p}q. \]

We then define another homomorphism $\varphi _n :\mathbb{Z}P(\alpha, n) \rightarrow \mathbb{Z}P(\alpha, n)$ as follows.
For any $n$-component nanophrase $p$,
\[\varphi _n(p)=\sum_{q\triangleleft p}(-1)^{\delta{(p,q)}}q. \]
We then extend $\varphi _n$ linearly to all of $\mathbb{Z}P(\alpha, n)$.

\begin{proposition} \label{f.pro1}
The homomorphism $\theta _n$ is an isomorphism and its inverse is given by $\varphi_n$.
\end{proposition}
\begin{proof}
For any $n$-component nanophrase $p$, 
\begin{align}
\varphi_n (\theta_n (p)) &= \varphi_n (\sum_{q\triangleleft p}q) \notag \\
&= \sum_{q\triangleleft p} \sum_{r\triangleleft q}(-1)^{\delta{(q,r)}}r \notag \\
&= \sum_{r\triangleleft p} \Bigl( \sum_{r \triangleleft q \triangleleft p} (-1)^{\delta{(q,r)}}r \Bigl) \label{f.prop1} \\
&= \sum_{r\triangleleft p} \Bigl( \sum_{i=0}^{\delta{(p,r)}} \binom {\delta{(p,r)}}{i} (-1)^{\delta{(p,r)}-i} r \Bigl)  \notag \\
&= p. \notag 
\end{align}

The forth line of (\ref{f.prop1}) means that we take a sum of $q$ satisfying $r \triangleleft q \triangleleft p$.
Here $i$ is the difference between the rank of $p$ and $q$.

Thus $\varphi_n (\theta_n (p))=p$ and so $\varphi_n \circ \theta_n$ is the identity map.
Similarly,
\begin{align*}
\theta_n (\varphi_n (p)) &= \theta_n \Bigl(\sum_{q\triangleleft p}(-1)^{\delta{(p,q)}}q \Bigr) \\
&= \sum_{q\triangleleft p}(-1)^{\delta{(p,q)}} \Bigl(\sum_{r\triangleleft q}r \Bigr) \\
&= p. 
\end{align*}
Thus $\theta_n \circ \varphi_n$ is also the identity map.
Therefore $\theta _n$ is an isomorphism and its inverse is given by $\varphi_n$.
\end{proof}

\begin{proposition}
The map $\theta _n$ induces an isomorphism $\hat{\theta } _n :G(\alpha , \nu ,n) \rightarrow G(\alpha , \nu ,n)$.
\end{proposition}

\begin{proof}
For any $w|AxAy|t-w|xByB|t$ in $\mathbb{Z}P(\alpha, n)$, we have

\begin{align}
\theta_n & (w|AxAy|t-w|xByB|t) \notag \\
&= \sum_{q\triangleleft AxAy}w|q|t - \sum_{q\triangleleft xByB}w|q|t \notag \\
&= \Bigl(\sum_{AA\triangleleft q\triangleleft AxAy}w|q|t + \sum_{q\triangleleft xy}w|q|t \Bigr) \label{f.prop2} \\
& \hspace{2em}- \Bigl(\sum_{BB\triangleleft q\triangleleft xByB}w|q|t + \sum_{q\triangleleft xy}w|q|t \Bigr)  \notag \\
&= \sum_{AA\triangleleft q\triangleleft AxAy}w|q|t - \sum_{BB\triangleleft q\triangleleft xByB}w|q|t.  \notag 
\end{align}
The last line of (\ref{f.prop2}) can be written by a finite sum of the form $w|AxAy|t-w|xByB|t$.
Similarly $\theta_n(w|Az|t-w|zA|t)$ can be written by a finite sum of the form $w|Az|t-w|zA|t$.
Therefore $\hat{\theta }_n$ is a homomorphism and so is $\hat{\varphi}_n$. 
Thus $\hat{\theta }_n$ is an isomorphism. 
\end{proof}

For each non-negative integer $m$, we define a map $O_m:P(\alpha, n) \rightarrow P(\alpha, n)$ by 
\begin{equation}
O_m(p)=\begin{cases}
p \hspace{2cm} \text{ if rank}(p) \leq m, \\
0 \hspace{2cm} \text{ otherwise}. 
\end{cases}
\end{equation}
We extend $O_m$ linearly to all of $\mathbb{Z}P(\alpha, n)$.

We introduce the following relation on $G(\alpha ,\nu ,n)$, 
\[p=0\]
if $p$ is an $n$-component nanophrase with rank($p$) $>m$.

Let $G_m(\alpha ,\nu ,n)$ be the group obtained from $G(\alpha ,\nu ,n)$ by taking quotient with the above relation which depends on $m$.  
Then $G_m(\alpha ,\nu ,n)$ is generated by the set of $n$-component nanophrases with rank $m$ or less.
As this set is finite, $G_m(\alpha ,\nu ,n)$ is a finitely generated abelian group.
Clearly $O_m$ induces a homomorphism $\hat{O}_m:G(\alpha ,\nu ,n) \rightarrow G_m(\alpha ,\nu ,n)$.
  
Let $\Gamma _{m,n}$ be the composition $\hat{O}_m \circ \hat{\theta } _n \circ \hat{I}^{-1}$ of $\hat{I}^{-1}$, $\hat{\theta } _n$ and $\hat{O}_m$.
Then $\Gamma _{m,n}$ is a homomorphism from $\mathbb{Z}P(\alpha,\nu ,n)$ to $G_m(\alpha ,\nu ,n)$.

The main theorem of this section is as follows.
\begin{theorem}\label{f.thm}
The map $\Gamma _{m,n} : \mathbb{Z}P(\alpha,\nu ,n) \rightarrow G_m(\alpha ,\nu ,n)$ is the universal finite type invariant of degree $m$ for cyclic equivalence classes of nanophrases.
\end{theorem}

\begin{corollary}
For any finite type invariant for $P(\alpha,\nu ,n)$ taking values in an abelian group $G$,
there is a finite subset $H$ of $G$ such that for any $p$ in $P(\alpha,\nu ,n)$ $v(p)$ is represented by a linear combination of elements of $H$.
\end{corollary}

We prepare lemmas for the proof of the theorem.

For any subphrases $p_1$ and $p_2$ of $p$, 
let $p_1 \setminus_p  p_2$ be the subphrase obtained from $p_1$ by deleting the letters contained in $p_2$.
\begin{example}
Let $p=AADBCBCD$, $p_1=AABB$ and $p_2=BCBC$.
Then $p_1 \setminus_p  p_2=AA$.
\end{example}
\begin{lemma}\label{lem2-1}
Let $p$ be an $n$-component nanophrase.
The sum of all the subphrases of $p$ regarded as an element of $\mathbb{Z}P(\alpha,n)$ can be written for any subphrase $t$ of $p$ as follows.
\begin{align}
\sum_{q\triangleleft p}q = \sum_{t\triangleleft r \triangleleft p} \Bigl( \sum_{(r \setminus_p t) \triangleleft s \triangleleft r}s \Bigr). \label{r.l}
\end{align}
\end{lemma}
\begin{proof}
For any subphrase $q$ of $p$, the right hand side of the above identity (\ref{r.l}) contains $q$.
In fact, letting $r$ be $t\cup_p q$, we have $t\triangleleft r\triangleleft p$ and $(r\setminus_p t) \triangleleft q\triangleleft r$.
For the left hand side of (\ref{r.l}), the number of terms is $2^{\text{rank}(p)}$.
For the right hand side, the number of subphrases $r$ such that $t\triangleleft r\triangleleft p$ is $2^{\delta{(p,t)}}$.
On the other hand, the number of subphrases $s$ such that $(r\setminus t) \triangleleft s\triangleleft r$ does not depend on $r$ and is $2^{\text{rank}(t)}$.
Therefore the number of terms on the right hand side of (\ref{r.l}) is $2^{\text{rank}(p)}$.
Thus we conclude this lemma.
\end{proof}

\begin{lemma}\label{lem2-2}
Let $p$ be an $n$-component nanophrase.
For any subphrase $t$ of $p$,
\[ \sum_{t\triangleleft q \triangleleft p} (-1)^{\delta{(p,q)}} \Bigl( \sum_{t \triangleleft r \triangleleft q}r \Bigr) = p. \]
\end{lemma}
\begin{proof}
The proof can be obtained by a similar way to proposition \ref{f.pro1}, and we omit it.
\end{proof}

\begin{proof}[Proof of Theorem \ref{f.thm} ]
First of all, we prove that $\Gamma _{m,n}$ is a finite type invariant of degree less than or equal to $m$.
Let $\hat{\Gamma} _{m,n} :\mathcal{P}(\alpha,\nu ,n) \rightarrow G_m(\alpha ,\nu ,n)$ be an extension of $\Gamma _{m,n}$.
Let $p$ be an element in $P_k(\alpha,\nu ,n)$ where $k>m$. 
Let $p'$ be the non-singular phrase obtained from $p$ by replacing all singular letters of $p$ with the corresponding non-singular letters 
and $p''$  the subphrase of $p$ obtained from $p$ by deleting all singular letters.
Then by proposition \ref{f.pro0}, the definition of $\Gamma _{m,n}$ and lemma \ref{lem2-1}, we have
\begin{align*}
\hat{\Gamma} _{m,n}(p) 
&= \Gamma _{m,n} \Bigl( \sum_{p''\triangleleft q \triangleleft p'} (-1)^{\delta(p',q)}q \Bigr) \\
&= \hat{O}_m \circ \hat{\theta} \circ \hat{I}^{-1} \Bigl( \sum_{p''\triangleleft q \triangleleft p'} (-1)^{\delta(p',q)}q \Bigr) \\
&= \hat{O}_m \circ \hat{\theta} \Bigl( \sum_{p''\triangleleft q \triangleleft p'} (-1)^{\delta(p',q)}q \Bigr) \\
&= \hat{O}_m\Bigl( \sum_{p''\triangleleft q \triangleleft p'} (-1)^{\delta(p',q)} (\sum_{r\triangleleft q}r) \Bigr) \\
&= \hat{O}_m \Bigl( \sum_{p''\triangleleft q \triangleleft p'} (-1)^{\delta(p',q)} \bigl( \sum_{p''\triangleleft s \triangleleft q} \sum_{(s\setminus_s p'')\triangleleft u \triangleleft s} u \bigr) \Bigr). 
\end{align*}
Here we define a map $g:P(\alpha,\nu ,n) \rightarrow \mathbb{Z}P(\alpha,\nu ,n)$ as follows.
We fix a nanophrase $p''$ and then for any $s$ in $P(\alpha,\nu ,n)$ such that $p'' \triangleleft s$,
\[g_{p''}(s) = \sum_{(s\setminus_s p'') \triangleleft u \triangleleft s} u.  \]
We extend $g$ linearly to all of $\mathbb{Z}P(\alpha,\nu ,n)$ and denote it by $g$ again.
Thus we have
\begin{align}
\sum_{p''\triangleleft q \triangleleft p'} (-1)^{\delta(p',q)} \bigl( \sum_{p''\triangleleft s \triangleleft q} (\sum_{(s\setminus_s p'')\triangleleft u \triangleleft s} u) \bigr)  
&=\Bigl( \sum_{p''\triangleleft q \triangleleft p'} (-1)^{\delta(p',q)} \bigl( \sum_{p''\triangleleft s \triangleleft q} g_{p''}(s) \Bigr). \label{f.thm1}
\end{align} 
Since $g$ is linear, lemma \ref{lem2-2} implies that 
\begin{align*}
\text{the right hand side of (\ref{f.thm1})} &=g_{p''} \Bigl( \sum_{p''\triangleleft q \triangleleft p'} (-1)^{\delta(p',q)} \bigl( \sum_{p''\triangleleft s \triangleleft q}s \bigr) \Bigr)\\
&= g_{p''}(p') \\
&= \sum_{(p'\setminus_{p'} p'')\triangleleft u \triangleleft p'}u.
\end{align*} 
Since $p'\setminus_{p'} p''$ is the subphrase of $p'$ which has rank $k$ ($>m$), each phrase of the above term has rank more than $m$. 
Thus the image of this term under $\hat{O}_m$ is zero.

Secondly, we prove that $\Gamma _{m,n}$ is a finite type invariant of degree $m$.
For any $n$-component nanophrase $p$ without singular letters, 
let $p^*$ be the phrase obtained from $p$ by replacing all letters with the corresponding singular letters.
We then have
\begin{align*}
\hat{\Gamma} _{m,n}(p^*) 
&= \hat{O}_m (\sum_{p \triangleleft s \triangleleft p}s) \\
&= \hat{O}_m(p) \\
&= p.
\end{align*}
Therefore $\hat{\Gamma} _{m,n}(p^*)$ is not equal to zero, and so the degree of $\Gamma _{m,n}$ is $m$.
 
Finally, we prove the universality as follows.
For any finite type invariant $v$ of degree less than or equal to $m$ taking values in $H$, 
we will show that there exists a homomorphism $f$ from $G_m(\alpha ,\nu ,n)$ to $H$ such $v=f\circ \Gamma _{m,n}$.  
We extend $v$ linearly to all of $\mathbb{Z}P(\alpha, \nu, n)$ and denote it by $v$ again.
We define a map $f$ as follows.
For any $n$-component nanophrase $p$,
\[ f(p) = v\circ \hat{I}\circ \hat{\theta } _n^{-1}(q), \] 
where $q$ is in $G_m(\alpha ,\nu ,n)$ such that $\hat{O}_m(q)=p$.
Note that $q$ is the sum of phrases.
Then $f$ is well defined.
In fact, for any $q$ and $q'$ such that 
\[ \hat{O}_m(q) =\hat{O}_m(q') =p,\] 
$\hat{O}_m(q-q')$ is equal to zero in $G_m(\alpha ,\nu ,n)$ and so $q-q'$ consists of phrases with rank more than $m$.
We then have 
\begin{align*}
\hat{v}(q^*) 
&= v \Bigl( \sum_{r \triangleleft q} (-1)^{\delta{(q,r)}}r \Bigr) \\
&= v\circ I \circ \hat{\theta}_n^{-1} (q).  
\end{align*}
By the above and the fact that $v$ is a finite type invariant of degree less than or equal to $m$, 
\begin{align*}
v\circ I \circ \hat{\theta}_n^{-1}(q-q') &= \hat{v}(q^*-q'^*) \\
&=0,
\end{align*}
where $q^*$ is the sum of phrases obtained from $q$ by replacing  all letters of each phrase of $q$ with the corresponding singular letter, and similarly $q'^*$.
Therefore $f$ is well defined.
We then obtain
\begin{align*}
v &= f \circ \hat{O}_m\circ \hat{\theta } _n \circ  \hat{I}^{-1} \\
&= f \circ \Gamma _{m,n}. 
\end{align*}
Therefore we conclude the universality.  
\end{proof}

For any nanophrase $w$, we denote by $[w]$ the sum 
of the all nanophrase which is cyclic equivalent to $w$.
For example, we consider the case when $\alpha=\{ \pm \}$ and $\nu$ sends $+$ to $-$. 
If $w = XX\bar{Y}\bar{Y}$ where $X$ and $\bar{Y}$ mean a letter $X$ with $|X|=+$ and $Y$ with $|Y|=-$ respectively, then 
$[w] = XX\bar{Y}\bar{Y} + YXXY + \bar{Y}\bar{Y}XX + \bar{X}\bar{Y}\bar{Y}\bar{X}$.
If $v = XXYY$, then
$[w] = XXYY + \bar{Y}XX\bar{Y}$.

For two arbitrary nanophrases $w$ and $v$, we define $\langle ,\rangle$ by
\[ \langle w, v \rangle = \#\{ {\rm subwords\hspace{.2em}of \hspace{.2em}}v \hspace{.2em}
{\rm isomorphic \hspace{.2em}to} \hspace{.2em} w \}. \]

We extend $\langle ,\rangle$ bilinearly, so that it is a map 
$\mathbb{Z}P(\alpha,\nu ,n) \times \mathbb{Z}P(\alpha,\nu ,n) \rightarrow \mathbb{Z}$.

\begin{lemma} \label{f.t.i}
For any $w$ in $P(\alpha ,n)$, $\langle [w], ~~\rangle : P(\alpha,\nu ,n) \rightarrow \mathbb{Z}$ is a finite type invariant of degree rank$(w)$.
\end{lemma}
\begin{proof}
It is easy to check that it is a homomorphism.
Suppose that $w$ is a nanophrase whose rank is $m$.
Let $p$ be an $n$-component nanophrase with rank($p$) $>m$.
Let $u(\cdot ) = \langle [w], \cdot  \rangle$.
By proposition \ref{f.pro0}, 
\begin{align}
\hat{u}(p)= \sum_{p''\triangleleft q \triangleleft p'} (-1)^{\delta{(p',q)}}u(q),  \label{sum}
\end{align} 
where $p'$ is the non-singular phrase obtained by replacing all singular letters of $p$ by the corresponding non-singular letters,  
and $p''$ is the subphrase of $p$ obtained from $p$ by deleting all singular letters.
If there is a subphrase $r$ of $p'$ whose rank is $m$, then let $k$ be the rank of the intersection of $r$ and $p''$.
Note that $k$ is at least 0 and at most $m$.
The number of $q$'s satisfying $r \triangleleft q$ and $p'' \triangleleft q \triangleleft p'$ is
\[ \sum_{s=0}^{k+1}  \binom{k+1}{s} (-1)^{s} = 0. \]
So $\langle [w], p \rangle=0$.
Since $\langle [w], w^* \rangle = \langle [w], w \rangle=1$, It is a finite type invariant of degree $m$.

\end{proof}

Let $\varphi_{m,n}$ denote a map
\[ P(\alpha, \nu, n) \rightarrow 
\bigoplus_{v \in P(\alpha,\nu ,n),  {\rm rank}(v) \leq m} \mathbb{Z}\langle v \rangle, \] 
defined by the following.
$\varphi_{m,n}(w)$ is a direct sum of $\langle [v], w\rangle$
for $v$ in $P(\alpha, \nu ,n)$ whose rank is less than or equal to $m$, where $\mathbb{Z}\langle v \rangle$ is an infinite cyclic group generated by $v$.
We extend $\varphi_{m,n}$ linearly to all of $\mathbb{Z}P(\alpha, \nu ,n)$.

\begin{proposition}
There exists an isomorphism $f$ from $\bigoplus_{v \in P(\alpha,\nu ,n),  {\rm rank}(v) \leq m} {\mathbb{Z}\langle v \rangle}$ to $G_m(\alpha,\nu,n)$ such that $f \circ \varphi_{m,n} = \Gamma _{m,n}$.
\end{proposition}

\begin{proof}
Since it is easy to check this proposition, we omit the proof.
\end{proof}

\begin{remark}
In \cite{label8048}, Fujiwara showed that when $\alpha=\{+, -\}$, this finite type invariant is a complete invariant for cyclic equivalence classes of signed words, so a complete invariant for stably homeomorphism classes of curves on closed oriented surfaces.  
\end{remark}

\begin{remark}
We compare $\Gamma _{m,1}$ with Ito's $SCI_m$.
Let $c$ be a generic immersed curve with $n$ crossings, and $w_c$ a signed word corresponding to $c$.
In \cite{MR2573961} Ito defined a map $SCI_m(c): W_m \rightarrow k$ where $k$ is a field and $W_m$ is a $k$-linear space generated by the isomorphism classes of signed words with rank $m$. 
$SCI_m(c)$ has an information of all subwords of $w_c$ with rank $m$.    
Thus when $\alpha$ is $\{\pm\}$, $\nu$ a map from $+$ to $-$, 
that is signed words, for any signed words $w_1$ and $w_2$ corresponding to generic immersed curves $c_1$ and $c_2$ respectively, $\Gamma _{m,1}(w_1) = \Gamma _{m,1}(w_2)$ if and only if $SCI_i(c_1) = SCI_i(c_2)$ for any $i$ between 0 and $m$.

\end{remark}

\section{Example}
Here is an example of finite type invariants.

\begin{example}
The rank of a nanophrase over $\alpha$ is a finite type invariant of degree $1$.
\end{example}    

\begin{example}
Let $\pi _{ij}$ ($1 \leq i <j \leq n$) be a free abelian group generated by elements in $\alpha $.
Let $\pi _{ii}$ ($1 \leq i \leq n$) be a free abelian group generated by elements in $\alpha $ with the relations, $a = \nu (a)$, for all $a$ in $\alpha $.

For an $n$-component nanophrase $p$, the {\it linking matrix} of $p$ is defined as follows. 
Let $\mathcal{A}_{ij} (p)$ ($1 \leq i <j \leq n$) be the set of letters which have one occurrence in the $i$th component of $p$ and the other occurrence in the $j$th component of $p$. 
Let $\mathcal{A}_{ii} (p)$ ($1 \leq i \leq n$) be the set of letters which have two occurrences in the $i$th component of $p$. 
For any $i$ and $j$, let $l_{ij}(p)$ be defined as follows.
\[l_{ij}(p) =\sum _{A\in \mathcal{A}_{ij} (p)} |A| \in \pi _{ij}. \]
Let the linking matrix $L(p)$ be the symmetric $n\times n$ matrix given by the entry $l_{ij}(p)$.
It is easy to see that the linking matrix of a nanophrase is an invariant for cyclic equivalence classes of nanophrases. 
\end{example}  

For any $p$ in $P(\alpha, \nu ,n)$, we define a map $\iota:P(\alpha, \nu ,n) \rightarrow \mathbb{Z}$ by $\iota(p)=1$ and extend the map linearly.

\begin{theorem}
The linking matrix is a finite type invariant of degree $1$, 
and a direct sum of the linking matrix and $\iota$ is the universal finite type invariant of degree $1$.
\end{theorem}

\begin{proof}
We first prove that the linking matrix $L(p)$ is a finite type invariant of degree $1$.
Let $p_{A^*B^*}$ be a nanophrase with two singular letters, $A^*$ and $B^*$.
Let $p_{AB}$ be the nanophrase obtained from $p$ by replacing $A^*$ and $B^*$ with $A$ and $B$, respectively. 
Let $p_{A}$ be the nanophrase obtained from $p$ by replacing $A^*$ with $A$ and deleting $B^*$.
Let $p_{B}$ be the nanophrase obtained from $p$ by replacing $B^*$ with $B$ and deleting $A^*$.
Let $p$ be the nanophrase obtained from $p$ by deleting $A^*$ and $B^*$.
We then can obtain that
\begin{align*}
L(p_{A^*B^*}) &= L(p_{AB}) - L(p_{A}) - L(p_{B}) + L(p).
\end{align*}
We show that $L(p_{A^*B^*})$ is zero.
It is sufficient to prove that $l_{ij}(p_{A^*B^*})$ is zero.  
We first consider the case that $i = j$.
Let the letters $A$ and $B$ appear twice in $i$th component.
We then obtain that 
\begin{align*}
l_{ii}(p_{AB}) &= l_{ii}(p) + |A| + |B| \\
l_{ii}(p_{A}) &= l_{ii}(p) + |A| \\
l_{ii}(p_{B}) &= l_{ii}(p) + |B|.
\end{align*}
Therefore $l_{ii}(p_{A^*B^*})$ is zero. 
Let the letter $A$ appear twice in $i$th component and the letter $B$ do not appear twice in $i$th component.
We then obtain that 
\begin{align*}
l_{ii}(p_{AB}) &= l_{ii}(p_{A}) = l_{ii}(p) + |A| \\
l_{ii}(p_{B}) &= l_{ii}(p). 
\end{align*}
Hence $l_{ii}(p_{A^*B^*})$ is zero. 
Let both the letters $A$ and $B$ do not appear twice in $i$th component.
We then obtain that
\begin{align*}
l_{ii}(p_{AB}) = l_{ii}(p_{A}) = l_{ii}(p_{B}) = l_{ii}(p). 
\end{align*}
Therefore $l_{ij}(p_{A^*B^*})$ is zero. 
The same conclusion can be obtained for the case that $i \neq j$.
Let $q_*$ be a nanophrase with only one singular letter.
If $i$th and $j$th components contain the singular letter, then $l_{ij}(q_*) \neq 0$.
Thus $L$ is a finite type invariant of degree $1$.

The proof is completed by showing the universality.
Our proof starts with the observation that 
\[ G_1(\alpha ,\nu ,n) \cong  \bigoplus_{1 \leq i < j \leq n} \pi_{ij} \oplus \bigoplus_{1 \leq i  \leq n} \pi_{ii} \oplus \mathbb{Z}. \]
We note that a nanophrase with rank 1 is of the form $\emptyset|\ldots|\emptyset|A|\emptyset|\ldots|\emptyset|A|\emptyset|\ldots|\emptyset$  or $\emptyset|\ldots|\emptyset|AA|\emptyset|\ldots|\emptyset$.
Let us define $g_{ija}$ ($1 \leq i < j \leq n$) to be the nanophrase with the form $\emptyset|\ldots|\emptyset|A|\emptyset|\ldots|\emptyset|A|\emptyset|\ldots|\emptyset$ where $A$ appears in both $i$th and $j$th components and $|A|=a$ in $\alpha$,
and $g_{iia}$ ($1 \leq i \leq n$) to be the nanophrase with the form $\emptyset|\ldots|\emptyset|AA|\emptyset|\ldots|\emptyset$ where $A$ appears in $i$th component and $|A|=a$ in $\alpha$.
Then the generators of $G_1(\alpha ,\nu ,n)$ are $g_{ija}$ ($1 \leq i \leq j \leq n$) and $\emptyset _n$.
We consider the $\nu$-shift move.
Since the $\nu$-shift move sends $g_{iia}$ to $g_{ii \nu(a)}$ and $g_{ija}$ to itself,
 the relations of $G_1(\alpha ,\nu ,n)$ are $g_{iia} =g_{ii \nu(a)}$ for any $i$ and $a$.
Thus $G_1(\alpha ,\nu ,n)$ is isomorphic to the abelian group generated by $g_{ija}$ ($1 \leq i \leq j \leq n$) and $\emptyset_n$ with the relations $g_{iia} =g_{ii \nu(a)}$ for all $a$ in $\alpha$ and $1 \leq i \leq n$.
This abelian group is isomorphic to 
\[ \bigoplus_{1 \leq i < j \leq n} \pi_{ij} \oplus \bigoplus_{1 \leq i \leq n} \pi_{ii} \oplus \mathbb{Z}, \]
by a direct sum of maps from $a$ in $\pi_{ij}$ to $g_{ija}$ for any $i$ and $j$ and from 1 in $\mathbb{Z}$ to 1 in $\mathbb{Z}\langle \emptyset_n \rangle$,
and so $f \circ (L \oplus \iota) = \Gamma_{1,n}$.
\end{proof}

\begin{remark}
This linking matrix is different from Fukunaga's linking matrix in \cite{MR2845717} on diagonal components.
Fukunaga's linking matrix is a universal finite type invariant for homotopy classes of nanophrases, in \cite{MR2786675}.
But his linking matrix is not a universal finite type invariant for cyclic equivalence classes of nanophrases. 
\end{remark}

\section{Finite type invariants for spherical curves}

From now on, we work only on the set $P(\alpha, \nu, 1)$, 
where $\alpha = \{ \pm \}$ and $\nu$ is the map which sends $+$ to $-$,
that is signed words.

Let ${P(\alpha, 1)}_s$ be a subset of $P(\alpha, 1)$, each of which has the corresponding generic curve on a sphere.
Let ${P(\alpha,\nu, 1)}_s$ be the union of the set of cyclic equivalence classes of ${P(\alpha, 1)}_s$ and $P(\alpha, 1) \setminus  {P(\alpha, 1)}_s$.
Note that according to our definition of finite type invariants, when we extend an invariant $u$ to $\hat{u}$, we must extend the domain to ${P(\alpha,\nu, 1)}_s$.
Therefore we consider not the set of cyclic equivalence classes of ${P(\alpha, 1)}_s$ but ${P(\alpha,\nu, 1)}_s$.

In \cite{MR2726565}, Ito introduced regular homotopy moves on signed words. 
We review these definition.  

\begin{definition}[Regular homotopy moves on signed words]
Three regular homotopy moves on signed words are defined as follows. 

A first regular homotopy is $(\mathcal{A}, xyz) \to (\mathcal{A} \cup \{ A, B \}, xA^\pm B^\mp yA^\pm B^\mp z)$, 
where $x$, $y$ and $z$ are words without letters $A$ and $B$. 

A second regular homotopy is $(\mathcal{A}, xyz) \to (\mathcal{A} \cup \{ A, B \}, xB^\pm A^\mp yB^\pm A^\mp z)$, 
where $x$ and $y$ are words without letters $A$ and $B$. 

A third regular homotopy is $(\mathcal{A}, xA^\pm B^\pm yA^\pm C^\pm zB^\pm C^\pm t) \to (\mathcal{A}, xB^\pm A^\pm yC^\pm A^\pm zC^\pm B^\pm t)$, 
where $x$, $y$, $z$ and $t$ are words without letters $A$, $B$ and $C$.
\end{definition}

These moves are corresponding to moves of curves illustrated in Figure \ref{fig:move}.
\begin{figure}[h]
\begin{center}
\includegraphics[scale=0.9]{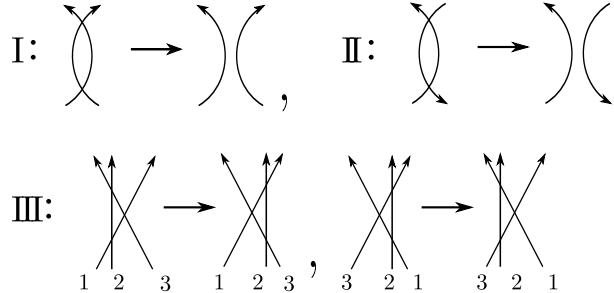}
\caption{Regular homotopy moves}
\label{fig:move}
\end{center}
\end{figure}

Arnold's basic invariants $J_s^+$, $J_s^-$ and $St_s$ are defined by their behavior under the moves and the inverse moves in Figure \ref{fig:move}.
In \cite{label8048}, Fujiwara extended Arnold's basic invariants of spherical curves to invariants on signed words, in \cite{MR1650406}. 
We review these definitions.  

\begin{definition}[Arnold's basic invariants on signed words]
$J_s^+$ decreases (respectively increases) by 2 under the first regular homotopy move (respectively inverse move), 
and does not change under the other moves.

$J_s^-$ increases (respectively decreases) by 2 under the second regular homotopy move (respectively inverse move), 
and does not change under the other moves.

$St_s$ increases (respectively decreases) by 1 under the third regular homotopy move (respectively inverse move), 
and does not change under the other moves.

Normalized conditions for Arnold's basic invariants are given as follows.
Let $w_i$  be a signed word $A_1A_1A_2A_2...A_iA_i$
 (or $\bar{A_1}\bar{A_1}\bar{A_2}\bar{A_2}...\bar{A_i}\bar{A_i}$).  
Then
\begin{align*}
J_s^+(w_i) = \frac{(i-1)^2}{2}, ~~
J_s^-(w_i) = \frac{(i-2)^2}{2} - \frac{3}{2}, ~~
St_s(w_i) = -\frac{(i-1)^2}{4}.
\end{align*}
\end{definition}

By using Polyak's formulation of the Arnold's basic invariants \cite{MR1650406},
we extend the Arnold's basic invariants to maps ${P(\alpha,\nu, 1)}_s \longrightarrow \mathbb{Z} $ defined by 
\begin{align*}
J_s^+(w) &= \langle \langle AABB - ABBA - 3ABAB, w \rangle \rangle  - \frac{1}{2} \langle [AA], w \rangle + \frac{1}{2} \\
J_s^-(w) &= \langle \langle AABB - ABBA - 3ABAB, w \rangle \rangle  - \frac{3}{2} \langle [AA], w \rangle + \frac{1}{2} \\
St(w) &= \frac{1}{2} \langle \langle - AABB + ABBA + ABAB, w \rangle \rangle  + \frac{1}{4} \langle [AA], w \rangle - \frac{1}{4},
\end{align*}
where a map $\langle \langle, \rangle \rangle$ is defined as follows.
For any Gauss word $v$ and signed word $w$,
we define $\langle \langle, \rangle \rangle$ by
\[ \langle \langle v, w \rangle \rangle = \sum_{w' \triangleleft w, w'\cong v} \prod _{A \text{ in } w', |A|=-1}(-1) \]
where $\cong$ means isomorphic as Gauss words.
We extend the map bilinearly.

\begin{theorem}
Arnold's basic invariants are finite type invariants of degree $2$.
But they do not provide the universal finite type invariant of degree $2$.
\end{theorem}

\begin{proof}
It is easy to see that $\langle \langle AABB - ABBA, ~~ \rangle \rangle$ is an invariant for ${P(\alpha,\nu, 1)}_s$.
By Polyak \cite{MR1650406} we see that $\langle \langle ABAB, ~~\rangle \rangle$ is also an invariant for ${P(\alpha,\nu, 1)}_s$.
Thus $J_s^+$, $J_s^-$ and $St$ are invariants for ${P(\alpha,\nu, 1)}_s$.
Similar to lemma \ref{f.t.i}, we have finite type invariants of degree 2.  

We compare a signed word $AA\bar{B}\bar{B}CC$ with $\bar{A}\bar{A}\bar{B}\bar{B}CC$.

\begin{figure}[h]
\begin{center}
\includegraphics[scale=0.2, angle=0]{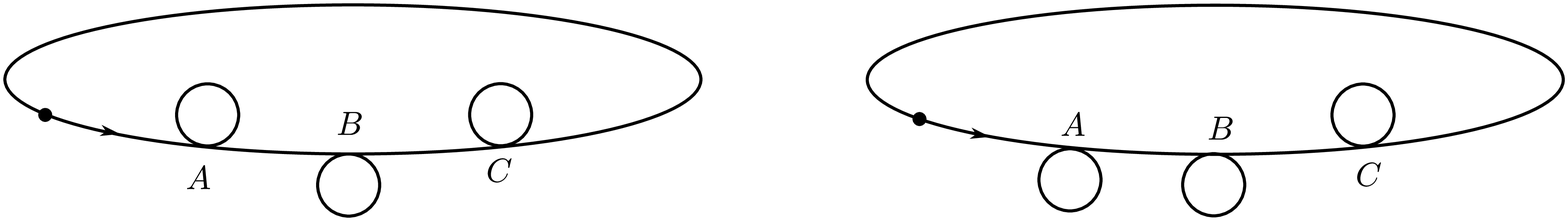}
\caption{}
\label{fig.e5}
\end{center}
\end{figure}\begin{align*}
J_s^+(AA\bar{B}\bar{B}CC) &= -2 = J_s^+(\bar{A}\bar{A}\bar{B}\bar{B}CC) \\
J_s^-(AA\bar{B}\bar{B}CC) &= -5 = J_s^-(\bar{A}\bar{A}\bar{B}\bar{B}CC) \\
St_s(AA\bar{B}\bar{B}CC) &= 1 =St_s(\bar{A}\bar{A}\bar{B}\bar{B}CC). 
\end{align*}
On the other hand, because $\varphi_{2,1}$ is an invariant for $P(\alpha, \nu, 1)$, $\varphi_{2,1}$ is also an invariant for ${P(\alpha, \nu, 1)}_s$, 
and we have
\begin{align*}
\varphi_{2,1}(AA\bar{B}\bar{B}CC) \neq \varphi_{2,1}(\bar{A}\bar{A}\bar{B}\bar{B}CC). 
\end{align*}
Therefore triple $J_s^+$, $J_s^-$ and $St$ do not provide the universal finite type invariant of degree 2.  

\end{proof}

\begin{remark}
Since $\langle [AABB], \rangle$, $\langle [AA\bar{B}\bar{B}], \rangle$, $\langle [\bar{A}\bar{A}\bar{B}\bar{B}], \rangle$, $\langle [ABAB], \rangle$, $\langle [AA], \rangle$, $\langle \emptyset, \rangle$ and $\langle \langle ABAB, \rangle \rangle$
are independent, the linear space generated by finite type invariants with degree at most 2 for spherical curves has dimension more than or equal to 7. 
\end{remark}

\section*{Acknowledgement}
The author thanks Professor Sadayoshi Kojima for his valuable
suggestions and comments.
The author would like to express her gratitude to Mikihiro Fujiwara.

\bibliography{kotorii}

\def\cprime{$'$} \def\cprime{$'$} \def\cprime{$'$}
\providecommand{\bysame}{\leavevmode\hbox to3em{\hrulefill}\thinspace}
\providecommand{\MR}{\relax\ifhmode\unskip\space\fi MR }
\providecommand{\MRhref}[2]{%
  \href{http://www.ams.org/mathscinet-getitem?mr=#1}{#2}
}
\providecommand{\href}[2]{#2}
\begin{thebibliography}{10}

\bibitem{MR1310595}
V.~I. Arnol{\cprime}d, \emph{Plane curves, their invariants, perestroikas and
  classifications}, Singularities and bifurcations, Adv. Soviet Math., vol.~21,
  Amer. Math. Soc., Providence, RI, 1994, With an appendix by F. Aicardi,
  pp.~33--91. \MR{1310595 (95m:57009)}

\bibitem{MR1286249}
\bysame, \emph{Topological invariants of plane curves and caustics}, University
  Lecture Series, vol.~5, American Mathematical Society, Providence, RI, 1994,
  Dean Jacqueline B. Lewis Memorial Lectures presented at Rutgers University,
  New Brunswick, New Jersey. \MR{1286249 (95h:57003)}

\bibitem{label8048}
M.~Fujiwara, \emph{Finite type invariants of words and {A}rnold's invariants},
  arXiv:math.GT/0808.3646. (2008).

\bibitem{MR2845717}
T~Fukunaga, \emph{Homotopy classification of nanophrases with at most four
  letters}, Fund. Math. \textbf{214} (2011), no.~2, 101--118. \MR{2845717}

\bibitem{MR2786675}
A~Gibson and N~Ito, \emph{Finite type invariants of nanowords and nanophrases},
  Topology Appl. \textbf{158} (2011), no.~8, 1050--1072. \MR{2786675
  (2012d:57016)}

\bibitem{MR1763963}
M.~Goussarov, M.~Polyak, and O.~Viro, \emph{Finite-type invariants of classical
  and virtual knots}, Topology \textbf{39} (2000), no.~5, 1045--1068.
  \MR{1763963 (2001i:57017)}

\bibitem{MR2573961}
N.~Ito, \emph{Finite-type invariants for curves on surfaces}, Proc. Japan Acad.
  Ser. A Math. Sci. \textbf{85} (2009), no.~9, 129--134. \MR{2573961}

\bibitem{MR2726565}
\bysame, \emph{Construction of invariants of curves and fronts using word
  theory}, J. Knot Theory Ramifications \textbf{19} (2010), no.~9, 1205--1245.
  \MR{2726565 (2012a:57017)}

\bibitem{MR1650406}
M.~Polyak, \emph{Invariants of curves and fronts via {G}auss diagrams},
  Topology \textbf{37} (1998), no.~5, 989--1009. \MR{1650406 (2000a:57080)}

\bibitem{MR2276346}
V.~Turaev, \emph{Knots and words}, Int. Math. Res. Not. (2006), Art. ID 84098,
  23. \MR{2276346 (2007k:57017)}

\bibitem{MR2352565}
\bysame, \emph{Topology of words}, Proc. Lond. Math. Soc. (3) \textbf{95}
  (2007), no.~2, 360--412. \MR{2352565 (2008i:57025)}

\bibitem{MR1089670}
V.~A. Vassiliev, \emph{Cohomology of knot spaces}, Theory of singularities and
  its applications, Adv. Soviet Math., vol.~1, Amer. Math. Soc., Providence,
  RI, 1990, pp.~23--69. \MR{1089670 (92a:57016)}

\end{thebibliography}
\bibliographystyle{amsplain}
\end{document}